\documentclass[12pt]{amsart}

\usepackage[frenchb, greek,english]{babel}
\usepackage[latin1]{inputenc}
\usepackage{amsmath,latexsym,amsmath}
\usepackage{amsfonts}
\usepackage{amssymb}
\usepackage{geometry}
\usepackage{graphicx}
\usepackage{amsthm}
\usepackage{bm}
\usepackage{setspace}
\usepackage[final]{showkeys}
\usepackage{dsfont}
\usepackage{enumitem}

\usepackage{epsfig}
\usepackage{graphicx}
\usepackage{psfrag}
\usepackage{enumitem}
\usepackage{bbm}
\usepackage{hyperref}

\theoremstyle{definition}
\newtheorem{The}{Theorem}
\newtheorem{Def}[The]{Definition}
\newtheorem{Cor}[The]{Corollary}

\newtheorem{Lem}[The]{Lemma}

\newtheorem{Rem}[The]{Remark}

\def\cv{\ensuremath{\text {Cor}}}
\newcommand{\dif}{\mathrm{d}}

\begin{document}

\title[Stretched exponential large deviations]{Large deviations for dynamical systems with stretched exponential decay of correlations}

\author[R. Aimino]{Romain Aimino}
\address{Romain Aimino\\ Centro de Matem\'{a}tica da Universidade do Porto\\ Rua do
Campo Alegre 687\\ 4169-007 Porto\\ Portugal}
\email{\href{mailto:romain.aimino@fc.up.pt}{romain.aimino@fc.up.pt}}
\urladdr{\url{http://www.fc.up.pt/pessoas/romain.aimino/}}

\author[J. M. Freitas]{Jorge Milhazes Freitas}
\address{Jorge Milhazes Freitas\\ Centro de Matem\'{a}tica \& Faculdade de Ci\^encias da Universidade do Porto\\ Rua do
Campo Alegre 687\\ 4169-007 Porto\\ Portugal}
\email{\href{mailto:jmfreita@fc.up.pt}{jmfreita@fc.up.pt}}
\urladdr{\url{http://www.fc.up.pt/pessoas/jmfreita/}}

\thanks{RA was partially supported by FCT grant SFRH/BPD/123630/2016. Both authors were partially supported by FCT projects FAPESP/19805/2014, PTDC/MAT-CAL/3884/2014 and PTDC/MAT-PUR/28177/2017, with national funds, and by CMUP (UID/MAT/00144/2019), which is funded by FCT with national (MCTES) and European structural funds through the programs FEDER, under the partnership agreement PT2020. RA would like to thank Jean-Ren\'e Chazottes and Cesar Maldonado for interesting discussions on this topic. JMF would like to thank Jos\'e Alves for useful comments. The authors would also like to thank the referees for their useful suggestions.}

\begin{abstract}
We obtain large deviations estimates for systems with stretched exponential decay of correlations, which improve the ones previously obtained in the literature. As a consequence we obtain better large deviations estimates for Viana maps and get large deviations estimates for a class of intermittent maps with stretched exponential loss of memory.
\end{abstract}

\maketitle

\section{Introduction}
\label{sec:introduction}

In the last two decades the study of statistical properties of non-uniformly dynamical systems has been capturing much interest and attention. The memory loss of the systems given in terms of decay of correlations and its connections with limiting laws, such as central limit theorems, invariance principles, extreme value distributions, and other properties such as large deviation principles have been investigated thoroughly. In order to prove such properties, several different techniques have been used, such as: inducing, coupling, spectral analysis of transfer operators or renewal equations. One of the tools that has revealed to be very powerful is the construction of Gibbs-Markov structures called Young towers whose inducing times allow to describe the rates of mixing of the system (\cite{Y98,Y99}) and ultimately obtain estimates for large deviations (\cite{MN08}) and several types of limiting laws.

In \cite{AFLV11}, the authors studied the relations between the rates of decay of correlations, large deviation estimates and the tail of the inducing times of Young towers. In particular, they proved a sort of reciprocal of Young's results, namely, they have shown that if the system has a certain rate of decay of correlations, then it admits a Young tower whose inducing times' tail decays at a similar rate. This was done both for local diffeomorphisms and maps with critical/singular sets. The construction of the Young tower and the estimates on the induced time tail follow from large deviations estimates for the expansion function and for the distance to the critical/singular set. Hence, at the core of that paper is a result (\cite[Theorem~D]{AFLV11}) that establishes a connection between the rates for large deviations estimates of observable functions and the rates of decay of correlations of those observables against essentially bounded functions. Two types of decay rates were considered: polynomial and stretched exponential. To be more precise, this result asserted that if a systems has decay of correlations for observables on a certain Banach space against all essentially bounded functions at a polynomial rate or at a stretched exponential rate, then the large deviations estimates for those particular observables decay, respectively, at a polynomial or stretched exponential rate. Moreover, the dependence of the exponents and constants was clearly stated. We remark that the polynomial case had already been proved in \cite{M09}. However, as a corollary from the stretched exponential counterpart, in \cite{AFLV11}, the authors obtained, for the first time, stretched exponential estimates for the large deviations of H\"older continuous observables evaluated along the orbits of Viana maps. Under certain certain more restrictive assumptions, in \cite[Theorem~E]{AFLV11}, the exponential case was also covered, but the assumption on the decay of correlations was rather very strong.

The main goal of this note is to improve the large deviations estimates for the stretched exponential case obtained in \cite[Theorem~D]{AFLV11}. Namely, we obtain a smaller loss in the exponent of the stretched exponential rate when one goes from decay of correlations to large deviations estimates. We apply this result to Viana maps in order to obtain the best rates of large deviations estimates for these maps available in the literature. We also apply it to some some non-uniformly expanding maps introduced in \cite{CDKM18}, which result from a modification of the intermittent maps studied in \cite{LSV99} carried out in order to produce stretched exponential decay of correlations. 

The proof of the result is based on a technique introduced by Gordin that allows to write the sum of the random variables generated by the dynamics as a sum of reverse time martingale differences plus a coboundary. Then the problem is reduced to  control the large deviations of the sum of martingale differences. In the polynomial case, in \cite{M09}, this is done using Rio's inequality. In the stretched exponential case, in \cite{AFLV11}, the Azuma-Hoeffding inequality was used, instead. Here, we use again Rio's inequality and a power series expansion of the exponential moments of the sum of martingale differences in order to improve the exponent in the large deviation estimates obtained in \cite[Theorem~D]{AFLV11} .  

This short paper is organised as follows. In Section~\ref{sec:results} we state the main result and give some applications. Section~\ref{sec:proofs} is dedicated to the proofs.     

\section{Statement of results and applications}
\label{sec:results}

Let $T : X \to X$ be a dynamical system with an ergodic invariant probability measure $\mu$. Let $\varphi : X \to \mathbb{R}$ be an observable with $\varphi \in L^1(\mu)$  and let $$\varphi_n = \sum_{k=0}^{n-1} \varphi \circ T^k.$$ 
Birkhoff's ergodic theorem guarantees that for $\mu$-a.e. $x\in X$, time averages converge to the spatial average, \emph{i.e.}, 
$$
\lim_{n\to\infty}\frac {1}{n}\varphi_n(x)=\int \varphi \,d\mu.
$$
The study of \emph{Large Deviations} concerns the probability of observing a deviation from the mean for Birkhoff averages, namely, for a deviation size $\varepsilon>0$, we define:
\begin{equation*}
\text{LD}(\varepsilon, \varphi,n)=\mu\left(\left|\frac {1}{n}\varphi_n-\int \varphi \,d\mu\right|>\varepsilon\right).
\end{equation*}
Of course, Birkhoff's ergodic theorem implies that for all $\varepsilon$, we have that $\text{LD}(\varepsilon, \varphi,n)$ vanishes as $n\to\infty$. Large Deviations theory concerns the rate at which this quantity goes to $0$. In the classical case of independent and identically distributed random variables with a finite exponential moment, $\text{LD}(\varepsilon, \varphi,n)$ decays exponentially fast in $n$ and, in fact, one can prove a Large Deviations principle which establishes that there exists a strictly convex function $I(\varepsilon)$, vanishing only at $0$, such that
$$
\lim_{n\to\infty}\frac1n\log(\text{LD}(\varepsilon, \varphi,n))=-I(\varepsilon).
$$
In this setting the function $I$ is also called Cr\'amer function and can be obtained from the knowledge of the distribution of random variable $\varphi$. Large Deviations principles have been proved by several authors for uniformly hyperbolic systems (see for example \cite{OP88,K90,L90,Y90,W96}). 

One of the main aspects of raised in \cite{M09, AFLV11}  was the intimate connection between Large Deviations and Decay of Correlations. 

\begin{Def}[Decay of correlations]
\label{def:dc}
Let \( \mathcal C_{1}, \mathcal C_{2} \) denote Banach spaces of real valued measurable functions defined on \( X \).
We denote the \emph{correlation} of non-zero functions $\phi\in \mathcal C_{1}$ and  \( \psi\in \mathcal C_{2} \) w.r.t.\ a measure $\mu$ as
\[
\cv_\mu(\phi,\psi,n):=\frac{1}{\|\phi\|_{\mathcal C_{1}}\|\psi\|_{\mathcal C_{2}}}
\left|\int \phi\, (\psi\circ T^n)\, \dif\mu-\int  \phi\, \dif\mu\int
\psi\, \dif\mu\right|.
\]

We say that we have \emph{decay
of correlations}, w.r.t.\ the measure $\mu$, for observables in $\mathcal C_1$ \emph{against}
observables in $\mathcal C_2$ if, for every $\phi\in\mathcal C_1$ and every
$\psi\in\mathcal C_2$ we have
 $$\cv_\mu(\phi,\psi,n)\to 0,\quad\text{ as $n\to\infty$.}$$
  \end{Def}

We say that we have \emph{decay of correlations against $L^1$
observables} whenever  this holds for $\mathcal C_2=L^1(\mu)$  and
$\|\psi\|_{\mathcal C_{2}}=\|\psi\|_1=\int |\psi|\,\dif\mu$.

For non-uniformly hyperbolic systems, decay of correlations may not be exponential and, in those cases, the decay rate of $\text{LD}(\varepsilon, \varphi,n)$ is not exponential. For example, in the intermittent case of the Manneville -- Pomeau maps, it has been proved that  $\text{LD}(\varepsilon, \varphi,n)$ decays polynomially fast (see for example \cite{MN08,PS09,M09, AFLV11}). In some cases, like for Viana maps, for which streched exponential decay of correlations has been proved, the rate of decay of $\text{LD}(\varepsilon, \varphi,n)$ is stretched exponential (\cite{AFLV11}). For other results regarding Large Deviations for non-uniformly hyperbolic systems, see for example: \cite{AP06,MN08,M09,PS09,AFLV11,V12,CTY17}.

Exponential decay of correlations alone has not proved to be enough to prove exponential Large Deviations. So far, exponential Large Deviations rates have only been obtained by requiring some stronger properties, like a spectral gap for the Perron-Frobenius operator or decay of correlations against all $L^1$ observations (see discussion in \cite{AFLV11}). There is some loss in the exponents when one goes from decay of correlations estimates to rates of Large Deviations, unless some more information is known. 

Very recently, in \cite{NT19}, the authors show that even for very well behaved systems with exponential decay of correlations, for unbounded observables, one cannot obtain Large Deviations rates better than stretched exponential. We consider observables in $L^\infty$ and we managed to improve the loss in the exponents observed from the rates of decay of correlations to Large Deviations rates. However, it remains an open question whether the loss we managed to improve here is optimal (as in the unbounded case) or not.

We are now ready to state our main result. In what follows, for notational simplicity, we assume without loss of generality that $\int_X \varphi d \mu = 0$. 

\begin{The} \label{pro:large_deviations} \em
Let $\varphi \in L^\infty(\mu)$. Suppose that there exist $C_\varphi>0$, $\tau>0$ and $\theta \in (0,1]$ such that 
\begin{equation} \label{eq:decay}
\left|\int_X \varphi . \psi \circ T^n d \mu \right| \le C_\varphi \| \psi \|_{L^\infty_\mu}  e^{-\tau n^\theta}, \: \: \forall \psi \in L^\infty(\mu), \: n \ge 0.
\end{equation}
Then there exists $c = c(\theta, \tau)>0$ such that for all $n \ge 1$ and $\epsilon >0$,
\[
\mu( |\varphi_n| > n \epsilon) \le 2 e^{- \tau' \epsilon^{2 \theta'} n^{\theta'}},
\]
with $ \theta' = \frac{\theta}{\theta + 1}$, $\tau' = c \widetilde{C}_\varphi^{-2 \theta'}$ and $\widetilde{C}_\varphi = \max\{ \| \varphi \|_{L^\infty_\mu}, C_\varphi\}$.
\end{The}

Note that the exponent $\theta'= \frac{\theta}{\theta + 1}$ that we obtain for the large deviations estimate is larger than the one obtained in \cite[Theorem~D]{AFLV11}, where the exponent was equal to $\frac{\theta}{\theta + 2}$.

Of course, this result gives rise to the natural question regarding to which extent the exponent for the stretched exponential large deviations estimate that we managed to improve here is optimal. 

\subsection{Large deviations for Viana maps}
\label{subsec:Viana-maps}
In \cite{V97}, Viana introduced an important class of nonuniform expanding dynamical
systems with critical sets in dimension greater than one.  This class of  maps
 can be described as
follows. Let $a_0\in(1,2)$ be such that the critical point $x=0$
is pre-periodic for the quadratic map $Q(x)=a_0-x^2$. Let
$S^1=\mathbb R/\mathbb Z$ and $b:S^1\rightarrow \mathbb R$ be a Morse function, for
instance, $b(s)=\sin(2\pi s)$. For fixed small $\alpha>0$,
consider the map
 \[ \begin{array}{rccc} \hat f: & S^1\times\mathbb R
&\longrightarrow & S^1\times \mathbb R\\
 & (s, x) &\longmapsto & \big(\hat g(s),\hat q(s,x)\big)
\end{array}
 \]
 where  $\hat q(s,x)=a(s)-x^2$ with
$a(s)=a_0+\alpha b(s)$, and $\hat g$ is the uniformly expanding
map of the circle defined by $\hat{g}(s)=ds$ (mod $\mathbb Z$) for some
integer $d\geq 2$. It is
easy to check that for $\alpha>0$ small enough there is an
interval $I\subset (-2,2)$ for which $\hat f(S^1\times I)$ is
contained in the interior of $S^1\times I$. Hence, $S^1\times I$ is a forward invariant region for any map $f$
sufficiently close to $\hat f$ in the $C^0$ topology. Any such map restricted to $S^1\times I$ is called a \emph{Viana map}.
It was shown in \cite{A00} that Viana maps have a unique ergodic absolutely continuous invariant probability measure (acip) \( \mu \).
In \cite{G06}, it was proved that every Viana map exhibits stretched exponential decay of correlations, with \( \theta=1/2  \), for H\"older continuous functions against \( L^{\infty}(\mu) \) functions.
Consequently, the following corollary is a direct application of Theorem~\ref{pro:large_deviations}.

\begin{Cor} \em
Let \( f \) be a Viana map and let \( \mu \) be its unique acip. Then, for every H\"older continuous observable \( \varphi \) and every \( \epsilon>0 \), there exist
 \(\tau=\tau(\varphi, \epsilon)>0 \) and \( C=C({\varphi, \epsilon})> 0 \) such that
\( \mu\left(\left|\frac{1}{n}\varphi_n-\int\varphi\, d\mu\right|> \epsilon\right) \leq C e^{-\tau n^{1/3}}.
 \)
\end{Cor}

For comparison purposes, we remark that the large deviations estimate obtained in \cite[Theorem~G]{AFLV11} was of the order of $e^{-\tau'n^{1/5}}$.
 
\subsection{Large deviations for intermittent maps with stretched exponential decay of correlations}
\label{subsec:Viana-maps}
We consider the family of interval maps $f_\gamma:[0,1]\to[0,1]$, with $\gamma\in(0,1]$, introduced in \cite[Appendix~A]{CDKM18}, which result from adapting the intermittent maps studied in \cite{LSV99} so that the contact between the graph of $f_\beta$ and the identity at the fixed point creates a stretched exponential fast accumulation of the pre-orbit of $1/2$ at $0$, which eventually is responsible for stretched exponential decay of correlations. Namely, consider that
\begin{equation}\label{eq:LSV-new} 
 f_\gamma(x) = \begin{cases}
    x\left(1 + \frac{(\log 2)^{\gamma^{-1} - 1}}{|\log x|^{\gamma^{-1} - 1}}\right), & x \leq 1/2 \\
    2 x - 1, & x > 1/2
  \end{cases}.
  \end{equation}
From \cite[Theorem~A.1]{CDKM18}, it follows that $f_\gamma$ has an absolutely continuous invariant measure $\mu$ and exhibits stretched exponential decay of correlations, with \( \theta=\gamma  \), for H\"older continuous functions against \( L^{\infty}(\mu) \) functions. Hence, as consequence of Theorem~\ref{pro:large_deviations}, we obtain:
 \begin{Cor} \em
Let \( f_\gamma \) be as in \eqref{eq:LSV-new} and let \( \mu \) be its acip. Then, for every H\"older continuous observable \( \varphi \) and every \( \epsilon>0 \), there exist
 \(\tau=\tau(\varphi, \epsilon)>0 \) and \( C=C({\varphi, \epsilon})> 0 \) such that
\( \mu\left(\left|\frac{1}{n}\varphi_n-\int\varphi\, d\mu\right|> \epsilon\right) \leq C e^{-\tau n^{\gamma/(\gamma+1)}}.
 \)
\end{Cor}

\section{Proof}\label{sec:proofs}

To prove Theorem~\ref{pro:large_deviations}, we will estimate all the moments of $\varphi_n$.

\begin{Lem} \label{lem:moments} \em
There exists $K = K(\theta, \tau)>0$ such that for all $q >0$ and $n \ge 1$,  and all $\varphi \in L^\infty(\mu)$ satisfying \eqref{eq:decay}, we have
\[
\left(\int_X | \varphi_n|^q d \mu \right)^{\frac{1}{q}} \le K \widetilde{C}_\varphi q^{\frac{1}{2} \left(1 + \frac{1}{\theta} \right)} n^{\frac 1 2}.
\]

\end{Lem}

\begin{proof} In this proof, we will use $K$ to designate a generic constant whose value can change from one occurence to the other. The value of $K$ depends only on $\theta$ and $\tau$, and in particular is independent from $n$, $q$ and the observable $\varphi$.

We will follow closely the proof of \cite[Lemma 2.1]{M09}, adapted to our assumption of stretched exponential decay, keeping a precise track of the dependence in $q$ of all the estimates.

Let $\mathcal{L} : L^1(\mu) \to L^1(\mu)$ be the transfer operator of $T$, {\em i.e.} the unique operator satisfying 
\[
\int_X \varphi . \psi \circ T d \mu = \int_X \mathcal{L} \varphi . \psi d \mu, \: \: \forall \varphi \in L^1(\mu), \: \forall \psi \in L^\infty(\mu).
\]
Since $L^\infty(\mu)$ is the dual of $L^1(\mu)$, \eqref{eq:decay} implies that $\| \mathcal{L}^n \varphi\|_{L^1_\mu} \le C_\varphi e^{-\tau n^\theta}$. Hence, for all $q \ge 1$, 
\[
\int_X | \mathcal{L}^n \varphi|^q d \mu \le \| \mathcal{L}^n \varphi\|_{L^\infty_\mu}^{q-1} \int_X | \mathcal{L}^n \varphi| d \mu \le C_\varphi \| \varphi \|_{L^\infty_\mu}^{q-1} e^{-\tau n^\theta}.
\]

Defining $\chi = \sum_{n=1}^\infty \mathcal{L}^n \varphi$, we get, for all $q \ge 1$,
\[
\begin{aligned}
\| \chi \|_{L^q_\mu} \le \sum_{n=1}^\infty \| \mathcal{L}^n \varphi\|_{L^q_\mu} 
\le \sum_{n=1}^\infty C_\varphi^{\frac 1 q} \| \varphi\|_{L^\infty_\mu}^{1 - \frac 1 q}e^{- \frac{\tau}{q} n^\theta}
 \le \widetilde{C}_\varphi \sum_{n=1}^\infty e^{- \frac{\tau}{q} n^\theta} 
& \le  \widetilde{C}_\varphi \int_0^\infty e^{- \frac{\tau}{q} t^\theta} dt \\
&=  \widetilde{C}_\varphi \frac{1}{ \theta} \left( \frac{q}{\tau}\right)^{\frac{1}{\theta}} \int_0^\infty e^{-s} s^{\frac{1}{\theta} - 1} ds \\
&=  \widetilde{C}_\varphi  q^{\frac{1}{\theta}} \frac{1}{\theta} \frac{1}{\tau^{\frac{1}{\theta}}} \Gamma \left(\frac{1}{\theta}\right) \\
&=  K \widetilde{C}_\varphi q^{\frac{1}{\theta}},
\end{aligned}
\]
where we have performed the change of variables $s = \frac{\tau}{q} t^\theta$.

Defining $\phi = \varphi + \chi - \chi \circ T$, we also have 
\[(
\| \phi \|_{L^q_\mu} \le \| \varphi \|_{L^q_\mu} + 2 \| \chi \|_{L^q_\mu} \le \| \varphi\|_{L^\infty_\mu} + 2 K \widetilde{C}_\varphi q^{\frac{1}{\theta}} \le K \widetilde{C}_\varphi q^{\frac{1}{\theta}}.
\]

It is immediate to verify that $\mathcal{L} \phi = 0$, and so $\{\phi \circ T^k; \, k = 0,1,2, \ldots\}$ is a sequence of reverse martingale differences. Passing to the natural extension, we can reduce the situation to the case where $\{\phi \circ T^k; \, k=0,1,2, \ldots \}$ is a sequence of martingale differences with respect to a filtration $\{ \mathcal{F}_k; \, k=0,1,2,\ldots \}$.

By Rio's inequality \cite[Theorem 2.5]{R00}, we have for all $q \ge 1$:
\[
\| \varphi_n \|_{L^{2q}_\mu}^{2q} \le \left( 4q \sum_{i=1}^n b_{i,n}\right)^q
\]
where
\[
b_{i,n} = \max_{i \le j \le n} \left\| \varphi \circ T^i \sum_{k=i}^j \mathbb{E}(\varphi \circ T^k | \mathcal{F}_i) \right\|_{L^q_\mu} \le \| \varphi\|_{L^\infty_\mu} \max_{i \le j \le n} \left\| \sum_{k=i}^j \mathbb{E}(\varphi \circ T^k | \mathcal{F}_i) \right\|_{L^q_\mu}
\]

From the definition of $\phi$ and the martingale property, it follows
\[
\sum_{k=i}^j \mathbb{E}(\varphi \circ T^k | \mathcal{F}_i) = \phi \circ T^i - \mathbb{E}(\chi \circ T^i | \mathcal{F}_i) + \mathbb{E}(\chi \circ T^{j+1} | \mathcal{F}_i),
\]
so that
\[
\left\| \sum_{k=i}^j \mathbb{E}(\varphi \circ T^k | \mathcal{F}_i) \right\|_{L^q_\mu} \le \| \phi \|_{L^q_\mu} + 2 \| \chi \|_{L^q_\mu} \le K \widetilde{C}_\varphi q^{\frac{1}{\theta}}.
\]
Hence $b_{i,n} \le \| \varphi\|_{L^\infty_\mu}K \widetilde{C}_\varphi q^{\frac{1}{\theta}} \le K \widetilde{C}_\varphi^2 q^{\frac{1}{\theta}}$ and we thus obtain for all $q \ge 1$
\[
\| \varphi_n \|_{L^{2q}_\mu}^{2q} \le K \left(  \widetilde{C}_\varphi^2q^{1 + \frac{1}{\theta}} n\right)^q,
\]
which yields the required estimate for all $q \ge 2$. The general case of $q >0$ is deduced by changing the value of the constant $K$, since the map $q \mapsto \left( \int_X | \varphi_n|^q d \mu\right)^{\frac{1 }{q}}$ is non decreasing for $q \in (0,\infty)$.
\end{proof}

\begin{Lem} \label{lem:exponential} \em
There exists $c = c(\theta, \tau)>0$ such that for all $\varphi \in L^\infty(\mu)$ satisfying \eqref{eq:decay}:
\[
 \sup_{n \ge 1} \int_X e^{\tau' n^{- \theta'} | \varphi_n|^{2 \theta'}} d \mu \le 2,
\]
with $\tau' = c \widetilde{C}_\varphi^{- 2 \theta'}$.
\end{Lem}

\begin{proof} By expanding the exponential in power series and using Fubini's theorem and Lemma~\ref{lem:moments}, we have for all $\tau'>0$ and $n \ge 1$:
\[
\int_X e^{\tau' n^{- \theta'} | \varphi_n|^{2 \theta'}} d \mu  = \sum_{k=0}^{\infty} \frac{(\tau')^kn^{-k  \theta'}}{k!}  \int_X | \varphi_n |^{2 k \theta'} d \mu  \le \sum_{k=0}^{\infty} \left(2 \theta' \tau' K^{2\theta'} \widetilde{C}_\varphi^{2 \theta'} \right)^k \frac{k^k}{k!}.
\]
Since, by induction over $k$, we have $k! \ge \left(\frac{k}{e} \right)^k$ for all $k \ge 1$, it follows
\[
\int_X e^{\tau' n^{- \theta'} | \varphi_n|^{2 \theta'}} d \mu \le \sum_{k=0}^{\infty} \left( 2e\theta'K^{2\theta'} \widetilde{C}_\varphi^{2 \theta'} \tau' \right)^k = 2,
\]
if we set $\tau' = c \widetilde{C}_\varphi^{-2\theta'}$ for $c = (4e \theta' K^{2\theta'})^{-1}$.
\end{proof}

\begin{proof}[Proof of Theorem~\ref{pro:large_deviations}] For all $n\ge 1$ and $\epsilon >0$, by Markov's inequality and Lemma \ref{lem:exponential}, 
\[
\begin{aligned}
\mu(| \varphi_n | > n \epsilon) & = \mu(e^{\tau'n^{-\theta'} |\varphi_n|^{2 \theta'}} > e^{\tau'n^{\theta'} \epsilon^{2 \theta'}}) \\
& \le e^{-\tau'n^{\theta'} \epsilon^{2 \theta'} } \int_X e^{\tau'n^{-\theta'}  | \varphi_n |^{2 \theta'}} d \mu \\
& \le 2 e^{-\tau'n^{\theta'} \epsilon^{2\theta'}}.
\end{aligned}
\]
\end{proof}

\begin{Rem} By changing the value of $c$, we can replace the constant $2$ by any constant of the form $1+ \delta$, $\delta >0$.
\end{Rem}

\bibliographystyle{amsalpha}


\bibliography{LDImprovement}

\end{document}